\documentclass[11pt]{article}

\newcommand\indlim\varinjlim

\newcommand{\wt}{\widetilde}

\newcommand\cA{\mathcal{A}}

\newcommand\cM{\mathcal{M}}

\newcommand\cO{\mathcal{O}}
\newcommand\cP{\mathcal{P}}

\newcommand\cV{\mathcal{V}}

\let\fg\undefined

\newcommand\fc{\mathfrak{c}}

\newcommand\fg{\mathfrak{g}}

\newcommand\ft{\mathfrak{t}}

\newcommand\fs{\mathfrak{s}}

\newcommand\fD{\mathfrak{D}}

\newcommand\fS{\mathfrak{S}}

\newcommand\bbA{\mathbb{A}}

\newcommand\GG{\mathbb{G}}

\newcommand\bA{{\bbA}}

\newcommand\bG{{\GG}}

\newcommand\A{\mathbf{A}}

\def\e{\mathrm{e}}

\newcommand\gl{\mathfrak{gl}}
\newcommand\GL{\mathrm{GL}}
\newcommand\SO{\mathrm{SO}}
\newcommand\SL{\mathrm{SL}}

\newcommand\Sp{\mathrm{Sp}}

\newcommand\tr{\mathrm{tr}}

\newcommand\Spec{\mathrm{Spec}}

\newcommand\Gm{\mathbb{G}_m}

\newcommand\Hom{\mathrm{Hom}}

\newcommand\Sym{\mathrm{Sym}}

\newcommand\ad{\mathrm{ad}}

\newcommand\sgn{\mathrm{sgn}}

\newcommand{\Fr}{\mathrm{Fr}}

\newcommand\git{{/\!\!/}}

\PassOptionsToPackage{dvipsnames,table}{xcolor}

\usepackage{amstext}
\usepackage{amssymb}
\usepackage{amsmath}
\usepackage{amsthm}

\usepackage{mathtools}

\usepackage[utf8]{inputenc}

\usepackage[T1]{fontenc}

\usepackage[english]{babel}

\usepackage[margin=0.9in]{geometry}
\usepackage{hyperref}

\usepackage{framed}

\usepackage{enumitem}

\usepackage[hang,flushmargin,multiple]{footmisc}

\usepackage{graphicx}

\usepackage{wrapfig}

\usepackage[all]{xy}

\usepackage{tikz}

\usepackage{tikz-cd}

\usetikzlibrary{matrix,calc,arrows}

\usepackage[dvipsnames,table]{xcolor}  
\usepackage{color}

\usepackage{soul}
\sethlcolor{yellow}

\usepackage{xargs}                      

\usepackage[colorinlistoftodos,prependcaption,textsize=tiny]{todonotes}

\newcommandx{\unsure}[2][1=]{\todo[linecolor=red,backgroundcolor=red!25,bordercolor=red,#1]{#2}}
\newcommandx{\change}[2][1=]{\todo[linecolor=blue,backgroundcolor=blue!25,bordercolor=blue,#1]{#2}}
\newcommandx{\info}[2][1=]{\todo[linecolor=OliveGreen,backgroundcolor=OliveGreen!25,bordercolor=OliveGreen,#1]{#2}}
\newcommandx{\improvement}[2][1=]{\todo[linecolor=Plum,backgroundcolor=Plum!25,bordercolor=Plum,#1]{#2}}
\newcommandx{\thiswillnotshow}[2][1=]{\todo[disable,#1]{#2}}

\newtheorem{theorem}{Theorem}
\newtheorem{lemma}[theorem]{Lemma}
\newtheorem{proposition}[theorem]{Proposition}
\newtheorem{definition}[theorem]{Definition}

\numberwithin{theorem}{section}

\title{The companion section for classical groups}
\author{T. Hameister, B. C. Ng\^o}
\date{}

\begin{document}


\maketitle

\begin{abstract}
	We use the companion matrix construction for $\GL_n$ to build canonical sections of the Chevalley map $[\fg/G]\to \fg\git G$ for classical groups $G$ as well as the group $G_2$. To do so, we construct canonical tensors on the associated spectral covers. As an application, we make explicit lattice descriptions of affine Springer fibers and Hitchin fibers for classical groups and $G_2$. 
\end{abstract}

\section{Introduction}

Let $G$ be a reductive group over $k$, and denote by $\fg$ its Lie algebra. The Chevalley map
\[
\chi\colon \fg\to \fg\git G,
\]
where $\fg\git G := \Spec (k[\fg]^G)$ denotes the invariant theoretic quotient of $\fg$ by the adjoint action of $G$, is of fundamental importance in the construction of the Hitchin system \cite{hitchin}. In particular, for $\fg = \mathfrak{gl}_n$, $\chi$ sends a matrix to its characteristic polynomial. 

In  \cite{kostant}, Kostant exhibited a section of the Chevalley map for a general reductive group $G$ under the assumption that the characteristic of $k$ does not divide the order of the Weyl group. Kostant's section was generalized in \cite{bouthier} and \cite{afv}, including the case of characteristics $p>2$ for classical groups and the group $G_2$. As explained in \cite{ngo}, this section can be used to construct sections of the Hitchin fibration and affine Springer fibers. However, Kostant's construction can be counter-intuitive for computations. To illustrate this latter point, consider the case $G = \GL_3(k)$, in which case $\fg\git G=\bA^3$ is the 3-dimensional affine space. The Kostant section is the map sending 
$$
(a_1,a_2,a_3)\in \fg\git G \quad \mapsto \quad \begin{pmatrix}
	\frac{a_1}{3} & \frac{a_1^2}{6}+\frac{a_2}{2} & -\frac{4a_1^3}{27}-\frac{a_1a_2}{3}-a_3 \\
	1 & \frac{a_1}{3} & \frac{a_1^2}{6}+\frac{a_2}{2} \\
	0 & 1 & \frac{a_1}{3}
\end{pmatrix}\in \fg
$$
If you introduced this problem to an undergraduate student of linear algebra, of course, they would not give you the answer above; they might instead suggest the map:
$$
(a_1,a_2,a_3)\in \fg\git G \quad \mapsto \quad \begin{pmatrix}
	0 & 0 & -a_3 \\
	1 & 0 & -a_2 \\
	0 & 1 & -a_1
\end{pmatrix}\in \fg
$$
sending a characteristic polynomial to its \emph{companion matrix}. The section to the Hitchin map that Hitchin constructed in \cite{hitchin} is not strictly the same as the one of \cite{ngo} in the sense that he does not rely on the Kostant section but another section that feels more like a generalization of the companion matrix. Instead of the companion matrix, a map $\fg\git G \to \fg$, we will construct a map $\fg\git G\to [\fg/G]$, where $[\fg/G]$ is the quotient of $\fg$ by the adjoint action of $G$ in the sense of algebraic stack. This section will be called the companion section, which is free of any choice. The present note aims to explicitly construct the companion section for classical groups, including the symplectic and orthogonal groups and $G_2$. As an application of the companion sections, we will give elementary descriptions of affine Springer fibers and Hitchin fibers for classical groups similar to the description of the Hitchin fibers in the linear case due to Beauville-Narasimhan-Ramanan. 

The emphasis of this work is on providing case-by-case explicit formulas for the companion section for classical groups. It is also possible to construct the companion section uniformly. This will be the subject of our subsequent work.

\section*{Acknowledgments}

The second author was partially supported by NSF grant DMS 2201314. The authors thank Alexis Bouthier for pointing out the literature on Kostant sections in small characteristics. Both authors thank the anonymous reviewer for helpful feedback.

\section{Tensors defining classical groups}

We will recall the standard definition of classical groups as the subgroup of the linear groups fixing certain tensors. This is very well known for symplectic and orthogonal groups but a bit less known for $G_2$, which in a certain respect could qualify as a classical group as well. 

Let $V$ be a $2n$-dimensional vector space over a base field $k$, $V^*$ its dual vector space. The linear group $\GL(V)$ acts on the space $\wedge^2 V^*$ of alternating bilinear forms on $V$ with an open orbit. An alternating bilinear form $\mu\in \wedge^2 V^*$ is considered non-degenerate if it lies in this open orbit. This is equivalent to requiring the induced map $\mu: V\to V^*$ to be an isomorphism. The stabilizer of such a non-degenerate alternating bilinear form is a symplectic group $G$. We note that $\mu \in \wedge^2 V^*$ is non-degenerate if $\wedge^n \mu \in \wedge^{2n} V^*$ is a non-zero vector of the 1-dimensional vector space $\wedge^{2n} V^*$ and as a result, $G$ is contained in the special linear group $\SL(V)$. Then, a $G$-bundle over a $k$-scheme $S$ consists of a locally free $\cO_S$-module $\cV$ of rank $2n$ equipped with an alternating bilinear form $\wedge^2_S \cV \to \cO_S$ which is non-degenerate fiberwise. Although the embedding of $G = \Sp_{2n}$ into $\GL_{2n}$ may differ by conjugation by an element of $\GL_{2n}$, as we are more concerned with $G$-bundles than $G$ itself, the specific choice of non-degenerate alternating form $\mu\in \wedge^2 V^*$ is immaterial. We will write $G=\Sp_{2n}$. 

Let $V$ be a $n$-dimensional vector space over a base field $k$, $V^*$ its dual vector space. The linear group $\GL(V)$ acts on the space $S^2 V^*$ of symmetric bilinear forms on $V$ with an open orbit. A symmetric bilinear form $\mu\in S^2 V^*$ is considered non-degenerate if it lies in this open orbit. This is also equivalent to the induced map $\mu: V\to  V^*$ being an isomorphism, which in turn is equivalent to the induced map $\wedge^n \mu:\wedge^n V\to \wedge^n V^*$ being an isomorphism of 1-dimensional vector spaces. We note that $\wedge^n V$ and $\wedge^n V^*$ are dual as vector spaces so that for every choice of a basis vector $\omega\in \wedge^n V$, we have a dual basis vector $\omega^* \in \wedge^n V^*$. A basis vector $\omega\in \wedge^n V$ is said to be compatible with $\mu$ if the equation $\wedge^n \mu (\omega)=\omega^*$ is satisfied. This equation has exactly two non-zero solutions $\omega\in \wedge^n V$, which differ by a sign. The stabilizer of a non-degenerate symmetric bilinear form $\mu \in S^2 V^*$ is an orthogonal group $O(\mu)$. The stabilizer of a pair $(\mu,\omega)$ consisting of a non-degenerate symmetric bilinear form $\mu \in S^2 V^*$ and a compatible basis vector $\omega\in \wedge^n V$ is the special orthogonal group $\SO(\mu,\omega)$ which is the neutral component of $O(\mu)$. We note that $\SO(\mu,\omega)= O(\mu)\cap \SL(V)$ so that the special orthogonal group can also be defined as the stabilizer of a pair $(\mu,\omega)$ as above but without requiring $\omega$ being compatible with $\mu$. The stabilizer of any such pair is a special orthogonal group $G$. A $G$-bundle over a $k$-scheme $S$ consists then in a locally free $\cO_S$-module $\cV$ of rank $n$ equipped a symmetric bilinear form $\wedge^2_S \cV \to \cO_S$ which is non-degenerate fiberwise. The embedding of $G = \SO_{n}$ into $\GL_n$ depends on the form $\mu$ and is well defined only up to conjugation by $\GL_n$. However, as we are more concerned with $G$-bundles than $G$ itself, choosing a specific non-degenerate symmetric form $\mu\in \wedge^2 V^*$ is immaterial. We will write $G=\SO_n$.

There is a simple tensor definition of $G_2$ due to Engel \cite{engel}. Let $V$ be a 7-dimensional vector space. The linear group $\GL(V)$ acts on the space $\wedge^3 V^*$ of non-degenerate trilinear forms on $V$ with an open orbit. We will follow Hitchin's \cite{hitchin3forms} in formulating the equation defining this open orbit . We will denote the contraction $ \wedge^3 V^* \times V \to \wedge^2 V^*$ by $\mu_v$ for $\mu\in \wedge^2 V^*$ and $v\in V$. For $v_1,v_2\in V$ and $\mu\in \wedge^3 V^*$, we then have
$$\mu_{v_1} \wedge \mu_{v_2} \wedge \mu \in \wedge^7 V^*.$$
By choosing a non-zero vector $\iota$ of the determinant $\wedge^7 V$, $\mu$ gives rise to a symmetric bilinear form $\nu\in S^2 V^*$
\begin{equation}
	\label{eqn: associated symmetric form}
	\nu(v_1,v_2)=\langle \iota,  \mu_{v_1} \wedge \mu_{v_2} \wedge \mu \rangle
\end{equation}
which is non-degenerate if and only if $\mu$ lies in the open orbit of $\wedge^3 V^*$. We will say that $\mu$ is a non-degenerate 3-form on $V$. The stabilizer of a non-degenerate 3-form is a group $G_2\ltimes \mu_3(k)$ where $\mu_3(k)$ is the group of 3rd roots of unity in $k$; We obtain the connected component, a group of type $G_2$, by taking the intersection with $SL(V)$. A $G_2$-bundle over a $k$-scheme $S$ is thus a locally free $\cO_S$-module $\cV$ of rank 7 equipped with an alternating trilinear form $\mu \in \wedge^3 \cV^*$ which is non-degenerate fiberwise together with a trivialization of the determinant. Again, a different choice of nondegenerate 3-form $\mu$ may give a $\GL_7$ conjugate embedding of $G_2$ into $\GL_7$. However, such a choice is immaterial for us.

\section{Spectral cover and the companion matrix}

For all groups $G$ discussed previously, including symplectic, special orthogonal, and $G_2$, $G$ is defined as a subgroup of $\GL_n$ fixing certain tensors. We call the inclusion $G\to \GL_n$ the standard representation of $G$. We also have the induced inclusion of Lie algebras $\fg\to \mathfrak{gl}_n$ compatible with the adjoint actions of $G$ and $\GL_n$. We derive a morphism between invariant theoretic quotients 
$$\fc=\fg\git G \to \gl_n \git \GL_n=\fc_n$$
which is a closed embedding for symplectic groups, odd special orthogonal groups, and $G_2$, but not for even orthogonal groups. For $\GL_n$, we have a spectral cover $\fs_n \to \fc_n$, defined in Section \ref{sec: GLn Case}, which is a finite flat morphism of degree $n$ so that $\cO_{\fs_n}$ is a locally free $\cO_{\fc_n}$-module of rank $n$ given with a canonical endomorphism $[x]$ which is the usual companion matrix. The main result of this work can be formulated as follows:

\begin{theorem}
	Let $G$ be a symplectic group, odd special orthogonal group, or $G_2$ group and $G\to \GL_n$ its standard representation. Let $\fc\to \fc_n$ be the induced map of Chevalley quotients which is a closed embedding in these cases. Then the restriction $\cO_{\fs_n}$ to $\fc$ 
	$$\cV=\cO_{\fc} \otimes_{\cO_{\fc_n}}\cO_{\fs_n}$$ 
	as locally free $\cO_{\fc}$-module affords a canonical tensor defining a $G$-reduction and the companion matrix for $\GL_n$ defines a canonical map $\fg\git G \to [\fg / G]$ which is a section of the natural map $[\fg/G] \to \fg \git G$. This statement  remains valid for even orthogonal groups after replacing $\fc\times_{\fc_n} \fs_n$ by its normalization.
\end{theorem}

We prove the theorem by a case-by-case analysis. In particular, we will construct the explicit tensors required in each case.

\subsection{Linear groups}
\label{sec: GLn Case}
We first recall how the companion matrix is connected to the universal spectral cover in the case $\GL_n$.  In this case, the Chevalley quotient $\fg\git G$ is the $n$-dimensional affine space $\A^n$ and the map $\chi:\fg \to \fg\git G$ is given by the characteristic polynomial $\chi(\gamma)=(a_1(\gamma),\ldots,a_n(\gamma))$ where $\gamma\in \fg$ and $a_i(\gamma)=(-1)^i \tr(\wedge^i \gamma)$.    
In this case we have $\fc_n=\Spec(A_n)$ where $A_n=k[a_1,\ldots,a_n]$. The spectral cover $\mathfrak{s}_n=\Spec(B_n)$ where $B_n$ is the $A_n$-algebra 
\[ B_n=A_n[x]/(x^n+a_1 x^{n-1} +\cdots +a_n)\]
which is a free $A_n$-module of rank $n$ as the images of $1,x,\ldots,x^{n-1}$ form an $A_n$-basis of $B_n$. 
We also note that $B_n$ is a regular $k$-algebra as it is isomorphic to the polynomial algebra of variables $a_1,\dots, a_{n-1},x$. On the other hand, $B_n$ is equipped with an $A_n$-linear operator $[x]:B_n\to B_n$ given by $b\mapsto bx$. To give a map $\fc_n \to [\gl_n/\GL_n]$ is equivalent to the data of a rank $n$ vector bundle $\mathcal{E}\to \fc_n$ together with a an $\cO_{\fc_n}$-linear endomorphism of $\mathcal{E}$; that is, at the level of modules, a free, rank $n$ $A_n$ module with an $A_n$ linear endomorphism. Hence, $B_n$ with the operator $[x]$ provides us with an $A_n$-point of $[\gl_n/\GL_n]$, and we have thus constructed a map $[x]:\fc_n \to [\gl_n/\GL_n]$ which is a section of $\chi:[\gl_n/\GL_n] \to \fc_n$.  In term of matrices with respect to the $A_n$-basis of $B_n$ given by $1,x,\ldots,x^{n-1}$, $[x]$ is given by the usual companion matrix
\begin{equation}\label{gamma}
	x_\bullet= \begin{pmatrix}
		0 &  & \cdots & -a_n \\
		1 & 0 & \cdots & -a_{n-1} \\
		& \cdots & 0 & \\
		0 &  & 1 & -a_1
	\end{pmatrix}\in \mathfrak{gl}_n(A)
\end{equation}
The companion matrix thus gives us a map $x_\bullet:\fc\to \fg$ in the case $G=\GL_n$ taking a point $a = (a_1,\dots, a_n)$ of $\fc$ to the matrix above. This construction is a section to the characteristic polynomial map. However, it is often more useful to think of $[x]$ as a map $[x]:\fc \to [\fg/G]$ in the case $G=\GL_n$.

Let $\fg$ come equipped with the homothety action of $\bG_m$ and $\fc$ with the induced action $t\cdot a_i = t^ia_i$. There is an issue with using the companion matrix to construct a section to the Hitchin map as the companion map $x_\bullet:\fc\to \fg$ is not $\Gm$-equivariant. We note, however, that the stack-valued map $[x]:\fc\to [\fg/G]$ is almost $\Gm$-equivariant in the sense that after a base change by the isogeny $\Gm\to \Gm$ given by $t\mapsto t^2$, it becomes equivariant because of the identity
\begin{equation} \label{conjugation}
	\mathrm{ad}(\mathrm{diag}(t^{n-1},t^{n-3},\ldots,t^{1-n}))( \gamma)
	= t^{-2} \begin{pmatrix}
		0 & 0 & \cdots & -t^{2n} a_n \\
		1 & 0 & \cdots & -t^{2n-2} a_{n-1} \\
		& \cdots & & \\
		0 &  & 1 & -t^2 a_1
	\end{pmatrix}. \end{equation} 
This explains why we have a section to the Hitchin map after choosing a square root of the canonical bundle as in \cite{hitchin}.

As we intend to use the companion matrix \eqref{gamma} to construct a canonical section to the Chevalley map $\chi:[\fg/G]\to \fc$ for classical groups, it is useful to further investigate the linear algebraic structure of $B_n$ as an $A_n$-module. We have a symmetric $A_n$-bilinear map $\xi:B_n\otimes_{A_n} B_n\to A_n$ given by $$\xi(b_1\otimes_{A_n} b_2)=\tr_{B_n/A_n}(b_1b_2)$$ 
thus an element $\xi \in S^2_{A_n} B_n^*$. Because this element induces degenerate forms over the ramification locus of $B_n$ over $A_n$, we need a correction term to get a symmetric bilinear form that is non-degenerate fiberwise. We will describe this correction and the associated nondegenerate form in Lemma \ref{lem:beta}.

The pairing $\xi$ defines an $A_n$-linear map $\mu:B_n\to B_n^*$ where $B_n^*=\Hom_{A_n}(B_n,A_n)$ and $\mu(b_1)(b_2) = \xi(b_1,b_2)$. We note that the $A_n$-module $B_n^*$ is naturally a $B_n$-module and $\mu: B_n\to B_n^*$ is $B_n$ linear; thus, it is uniquely determined by the image of $1\in B_n$ that we will also denote by $\mu \in B_n^*$.  We will show that $B_n^*$ is a free $B_n$-module of rank 1, construct a generator of $B_n^*$ and find an explicit formula for $\mu\in B_n^*$ as a multiple of this generator.

\begin{lemma}
	\label{lem:beta}
	Let us denote by $v_0,\ldots,v_{n-1}$ the basis of $B_n$ given by the images of $1,x,\ldots,x^{n-1}$ in $B_n$ and $v_0^*,\ldots,v_{n-1}^*$ the dual basis of $B_n^*$. Then $\beta^*=v_{n-1}^*$ is a generator of $B_n^*$ as a $B_n$-module. Let us denote $f'\in B_n=A_n[x]/(f)$ the image of the derivative \[nx^{n-1}+(n-1)a_1x^{n-2}+\cdots+a_{n-1}\in A_n[x]\] 
	of the universal polynomial $f=x^n +a_1 x^{n-1}+\cdots+a_n \in A_n[x]$. Then we have $\mu= f' \beta^*$. 
\end{lemma}

\begin{proof}
	First, the discriminant $d$ of the universal polynomial $f$, defined as the resultant between $f$ and its derivative, is a nonzero element of the polynomial ring $A_n$. Indeed, $d$ defines the ramification divisor of the finite flat covering $\fs\to \fc$, which is generically étale for there exist separable polynomials in $k'[x]$ of degree $n$ with coefficients in any infinite field $k'$ containing $k$.
	We denote $A'=A_n[d^{-1}]$ the localization of $A_n$ obtained by inverting $d$, and $B'=B_n\otimes_{A_n} A'$. By construction, $f'$ is an invertible element of $B'$. The trace map $\mathrm{tr}_{B'/A'}:B'\to A'$ of $B'$ as free $A'$-module of rank $n$ is now given by the Euler formula (cf. III.6, Lemma 2 in \cite{serre})
	\[
	\mathrm{tr}_{B'/A'}\left(\frac{x^k}{f'}\right) = \begin{cases}
		0 & \text{if }k<n-1\\
		1 & \text{if }k = n-1
	\end{cases}
	\]
	If $v_0,\ldots,v_{n-1}$ denote the basis of $B'$ given by the images of $1,x,\ldots,x^{n-1}$ in $B_n$ and $v_0^*,\ldots,v_{n-1}^*$ the dual basis of $(B')^*$, then we derive from the Euler formula that the identities
	\begin{equation} \label{Euler}
		\mu(v_i)= f' \left (v^*_{n-1-i}+ \sum_{j<i} a'_{i,j} v^*_{n-1-j} \right )
	\end{equation}
	hold in $B_n^*\otimes_{A_n} A'$ for some $a'_{i,j}\in A'$. In particular, we have $\mu(v_0)=f' v^*_{n-1}$.
	As the localization map $B_n^* \to B_n^*\otimes_{A_n} A'$ is injective, this identity also holds in $B_n^*$. It follows that $\mu=f' v_{n-1}^*$ as desired. \end{proof}

As a consequence, we have a canonical nondegenerate bilinear form $\beta^*: B_n\otimes_{A_n} B_n\to A_n$ which is symmetric with respect to which the $A_n$-linear operator $[x]: B_n\to B_n$ is anti-self-adjoint; that is, for all $v_1,v_2\in B_n$, we have
\[
\beta^*(xv_1,v_2)+\beta^*(v_1,xv_2) = 0.
\]

For $G=\SL_n$, the Lie algebra $\fg=\mathfrak{sl}_n$ is the space of traceless matrices. We have $\fc=\Spec(A)$ where $A=k[a_2,\ldots,a_n]$. We note that for $a_1=0$, the companion matrix \eqref{gamma} is traceless and thus gives rise to a $A$-point on $\mathfrak{sl}_n$. The companion map $\gamma:\fc\to \fg$ induces a map $[\gamma]:\fc\to [\fg/G]$. The latter lays over the point of $B G$ with values in $A$ corresponding to the $\SL_n$-bundle corresponding to rank $n$ vector bundle $B$ equipped with the trivialization of the determinant given by the basis $1,x,\ldots, x^{n-1}$. The formula \eqref{conjugation} shows that the map $[\gamma]:\fc\to [\fg/G]$ is equivariant with respect to the isogeny $\Gm\to \Gm$ given by $t\mapsto t^2$ for the diagonal matrix $\mathrm{diag}(t^{n-1},t^{n-3},\ldots,t^{1-n})$ belonging to $\SL_n$.

\subsection{Symplectic groups}
\label{subsection:symplectic}

In the case $G=\Sp_{2n}$, we have $\fc=\Spec(A)$ with $A=k[a_2,\ldots,a_{2n}]$. The spectral cover $\mathfrak{s}=\Spec(B)$ where  \[B=A[x]/(x^{2n}+a_2 x^{2n-2} +\cdots +a_{2n})\] is a free $A$-module of rank $2n$, is equipped with an involution $\tau:B\to B$ given $\tau(x)=-x$. The companion matrix \eqref{gamma} gives a $A$-linear endomorphism of $B$ as a free $A$-module. For the companion matrix to produce a section to the Chevalley map $[\fg/G]\to \fc$ in the symplectic case, we need to construct a canonical nondegenerate symplectic form $\omega$ on the $A$-module $B$ for which $\gamma$ is anti-self-adjoint in the sense that
\[
\omega(\gamma v_1,v_2)+\omega(v_1,\gamma v_2) = 0
\]
for all $v_1,v_2\in B$.

The standard representation $\Sp_{2n}\to \GL_{2n}$ induces a map on Chevalley bases $\fc\to \fc_{2n}=\Spec(A_{2n})$ where $A_{2n}=k[a_1,\ldots,a_{2n}]$ which identidies $\fc$ with the closed subscheme of $\fc_{2n}$ defined by the ideal generated by $a_1,a_3,\ldots,a_{2n-1}$. We have $B=A\otimes_{A_{2n}} B_{2n}$ where $B_{2n}$ is the finite free $A_{2n}$-algebra defining the spectral covering of $\fc_{2n}$. If we denote $B^*=\Hom_A(B,A)$ then we have $B^*=A\otimes_{A_{2n}}B^*_{2n}$ where $B^*_{2n}=\Hom_{A_{2n}}(B_{2n},A_{2n})$. The generator $\beta_{2n}^*$ of the free $B_{2n}$-module $B^*_{2n}$ defined in Lemma \ref{lem:beta} then induces a generator $\beta^*$ of $B^*$ as a free $B$-module of rank one which can also be viewed as the bilinear form $\beta^*:B\otimes_A B \to A$ given by $b_1\otimes_A b_2=\tr_{B/A}({f'}^{-1} b_1 b_2)$ after localization.

The bilinear form $\omega: B\otimes_A B\to A$
$$\omega(b_1,b_2)= \beta^*(b_1,\tau(b_2)) = \tr_{B/A}({f'}^{-1} b_1 \tau(b_2))$$
with the second identity only making sense after localization of $A$ making $f'$ invertible, is a non-degenerate symplectic form for which $[x]$ is anti-self-adjoint. Indeed, we have 
\[\omega(b_1,b_2)=-\omega(b_2,b_1)\] because $\tau(f')=-f'$ for $f'\in A[x]$ is an odd polynomial as $f\in A[x]$ is an even polynomial. The equation $\omega(xb_1,b_2)+\omega(b_1,xb_2) = 0$ can be derived from $\tau(x)=-x$.

It follows that we have a morphism 
$$[x]:\fc \to [\fg/G]$$
which deserves to be called the companion map for the symplectic group. To obtain a companion matrix $x_\bullet:\fc\to \fg$, it is enough to find a trivialization of the $G$-bundle associated with the non-degenerate symplectic form $\omega: B\otimes_A B\to A$. For most applications, particularly the Hitchin fibration, we only need the section $[x]:\fc\to [\fg/G]$.

\subsection{Odd special orthogonal groups}

In the case $G = \SO_{2n+1}$, we have $\fc = \Spec(A)$ with $A = k[a_{2},a_4,\dots,a_{2n}]$. The spectral cover is defined as $\fs = \Spec(B)$ where $B = A[x]/(f)$ with $f=x f_0$ and $f_0 = x^{2n}+a_2x^{2n-2}+\cdots +a_{2n}$. $B$ is a free $A$-module of rank $2n+1$. As in the symplectic case, we will define a symmetric non-degenerate bilinear form $B\otimes_A B\to A$ for which the multiplication by $x$ is anti-self-adjoint.

The standard representation $\SO_{2n+1}\to \GL_{2n+1}$ gives rise to a map $\fc\to \fc_{2n+1}=\Spec(k[a_1,\ldots, a_{2n+1}])$ which is a closed embedding defined by the ideal generated by $a_1,a_3,\ldots,a_{2n+1}$. We have $B=A\otimes_{A_{2n+1}}B_{2n+1}$ where $B_{2n+1}$ is the finite free $A_{2n+1}$-module of rank $2n+1$ defining the spectral cover in the case $\GL_{2n+1}$. We also have $B^*=A\otimes_{A_{2n+1}}B_{2n+1}^*$ where $B^*=\Hom_A(B,A)$ and $B_{2n+1}^*=\Hom_{A_{2n+1}}(B_{2n+1},A_{2n+1})$. Following the discussion in the linear case $B_{2n+1}^*$ is a free $B_{2n+1}$ generated by the element $\beta^*_{2n+1}=(f')^{-1} \mu$ where $\mu$ is the trace form $\mu(b_1\otimes_A b_2)=\tr(b_1b_2)$. It induces a generator $\beta^*$ of $B^*$ as a $B$-module. We define the bilinear form $\omega:B\otimes_A B\to A$ by 
\begin{equation}
	\label{eqn: odd orthogonal form}
	\omega(b_1,b_2)=\beta^*(b_1,\tau(b_2))=\tr({f'}^{-1} b_1\tau(b_2)).
\end{equation}
The bilinear form $\omega$ is a nondegenerate bilinear form because $\beta^*$ is. It is symmetric because $\tau(f')=f'$ as $f'$ is an even polynomial. The equation $\omega(xb_1,b_2)+\omega(b_1,xb_2)$ can be derived from the fact $\tau(x)=-x$.

By choosing a trivialization of the determinant, we obtain a companion map $[x]:\fc\to [\fg/G]$ for $G=\SO_{2n+1}$.

\subsection{Even special orthogonal groups}

The case $G = \SO_{2n}$ is slightly more difficult for the map $\fc\to \fc_{2n}$ induced by the standard representation of $\SO_{2n}$ is not a closed embedding. Indeed, we have $\fc_{2n}=\Spec(A_{2n})$ where $A_{2n}=k[a_1,\ldots,a_{2n}]$ but $\fc=\Spec(A)$ where $A=k[a_2,\ldots,a_{2n-2},p_n]$ where $p_n$ is the Pfaffian satisfying $p_n^2=a_{2n}$ does not lie in the image of $A_{2n}\to A$. If $B_{2n}$ is the spectral cover of $A_{2n}$ and $B=A\otimes_{A_{2n}}B_{2n}$ then we have  
$$B = A[x]/(x^{2n}+a_2x^{2n-2}+\cdots+a_{2n-2}x^2+p_n^2).$$
As indicated by Hitchin \cite{hitchin}, the true spectral cover for even special orthogonal groups is not $B$ but its blowup $\wt B$ along the singular locus defined by $x$. We have
$$\wt{B} = A[x,p_{n-1}]/\big(xp_{n-1}-p_n,\;  x^{2n-2}+a_2x^{2n-4}+\cdots+a_{2n-2}+p_{n-1}^2 \big)$$
which is a free $A$-module of rank $2n$ and smooth as a $k$-algebra. We have an involution $\tau$ on $B$ and $\wt B$ given by $\tau(x)=-x$ and $\tau(p_{n-1})=-p_{n-1}$.

The dualizing sheaf $\omega_{\wt{B}/A}$ is a free rank-one $\wt B$-module, canonically isomorphic to $\wt B$ away from the ramification locus. As a $\wt B$-submodule of $\Fr(\wt B)$ it is generated by the inverse of the different $\fD_{\wt B/A}$ which is given by the formula 
\begin{eqnarray*}
	\fD_{\wt B/A} &=& \det \begin{pmatrix}
		- p_{n-1} & f' \\ -x & 2 p_{n-1}
	\end{pmatrix}\\ 
	&=& (n-1)x^{2(n-1)}+ (n-2)a_2 x^{2(n-2)} + \cdots +a_{2n-2} x^2+p_{n-1}^2.
\end{eqnarray*}
In other words, the bilinear form $\wt B\otimes_A \wt B \to A$ given by 
$$b_1 \otimes_A b_2 \mapsto \tr_{\wt B/A} (\fD_{\wt B/A}^{-1} b_1 b_2)$$
is non-degenerate. As in the symplectic and odd special orthogonal cases, we now consider the symmetric bilinear form 
$$\omega(b_1,b_2)=\tr_{\wt B/A} (\fD_{\wt B/A}^{-1} b_1 \tau(b_2))$$
Then $\omega$ is a non-degenerate symmetric bilinear form because $\tau(\fD_{\wt B/A})=\fD_{\wt B/A}$, and it satisfies 
$$\omega(xb_1,b_2)=-\omega(b_1,xb_2).$$

After a choice of trivialization of the determinant of $\wt B$ as a free $A$-module of rank $n$, the multiplication by $x$ gives rise to the companion section $[x]:\fg \git G \to [\fg/G]$ for the odd special orthogonal group $G=\SO_{2n+1}$.

\subsection{The group $G_2$}
\label{sec: G2 case}
In the case $G_2$, the invariant quotient is $A = k[e,q]$ with $\deg(e) = 2$ and $\deg(q) = 6$. The spectral cover $\fs = \Spec(B)$ of $\fc=\Spec(A)$ given by 
$$B = A[x]/\left(xf_0\right)\quad\text{for }f_0 = x^6-ex^4+\frac{e^2}{4}x^2+q$$
is a reducible cover of $A$ with two components corresponding to the quotient maps
\[
B\to B' = A[x]/(f_0)\quad\text{and}\quad B\to A=A[x]/(x).
\]
The cover $\fs' = \Spec(B')$ of $A$ is finite, flat of degree 6, and factors through two subcovers, of degrees 2 and 3, corresponding to the sub-$A$-algebras
$$A\subset A[y]/\left(y^3-ey^2+\frac{e^2}{4}y+q\right)\subset B'\quad \text{where }y=x^2$$
$$A\subset A[z]/\left( z^2+q \right)\subset B'\quad \text{where }z = x\left(x^2-\frac{e}{2}\right)$$

Let $\epsilon\in B[q^{-1}]^*:=\Hom_{A[q^{-1}]}(B[q^{-1}],A[q^{-1}])$ be dual to $f_0$; $\delta_i\in B[q^{-1}]^*$ be dual to $x^{i}$; and $\eta_i\in B[q^{-1}]^*$ be dual to $x^iz$ for $i=1,2,3$.
Let ${\rm tr}_z$ denote the skew-symmetric bilinear form on $B$ given by
\[
{\rm tr}_z(g,h) = {\rm Tr}_{{\rm Frac}(B)/{\rm Frac}(A)}\left(\frac{g(x)h(-x)z}{f(x)}\right)
\]
We will denote by $\rho$ the 3-form on $B[q^{-1}]$ given by
\begin{equation}
	\label{eqn:G2form}
	\rho:=\delta_1\wedge \delta_2\wedge \eta_3 + \delta_1\wedge\eta_2\wedge \delta_3 + \eta_1\wedge\delta_2\wedge\delta_3 - q\cdot \eta_0\wedge \eta_1\wedge \eta_2 + \epsilon\wedge {\rm tr}_z
\end{equation}
A priori, the 3-form above is valued in $A[q^{-1}]$. The next proposition tells us that it restricts to an element of $\bigwedge^3_A B^*$.
\begin{proposition}
	Restricting the 3-form $\rho$ to $B\to B[q^{-1}]$ induces a 3-form $\rho\in \bigwedge^3_A B^*$. In other words, $\rho$ takes values in $A$ when restricted to $B$.
\end{proposition}
\begin{proof}
	Consider the $A$-basis of $B$ given by $$\{1,x^i,x^iz\colon i = 1,2,3\}.$$ 
	This differs from the $A[q^{-1}]$-basis $$\{f_0,x^i,x^iz\colon i = 1,2,3\}$$ of $B[q^{-1}]$ only by scaling $f_0$. As $\rho$ is valued in $A$ on the $A$-linear span of the latter basis, it suffices to check the contraction $\iota_1\rho$ of $\rho$ along $1\in B$ is valued in $A$. We compute
	\begin{align*}
		\iota_1\rho &= \eta_1\wedge\eta_2-\frac{e}{2}\eta_2\wedge \eta_3 + \frac{1}{q}({\rm tr}_z - \delta_1\wedge\delta_2+\frac{e}{2}\delta_2\wedge \delta_3 - i_{x^3z}{\rm tr}_z\wedge \epsilon + \frac{e}{2}i_{xz}{\rm tr}_z\wedge \epsilon) \\
		& = \eta_1\wedge\eta_2-\frac{e}{2}\eta_2\wedge \eta_3 + \Big[\frac{1}{q}({\rm tr}_z - \delta_1\wedge\delta_2+\frac{e}{2}\delta_2\wedge \delta_3) -\iota_1{\rm tr}_z\wedge \epsilon \Big] \\
	\end{align*}
	Rewriting the latter in terms of a dual basis $\xi_i$, $i=0,\dots,6$ for the $A$-basis $\{x^i\colon i=0,\dots,6\}$ of $B$, we see that the expression in square brackets above is
	\[
	\iota_1\rho = \e_3\wedge \e_6+\e_4\wedge \e_5 -\frac{3e}{2} \e_5\wedge \e_6
	\]
	whose image lies in $A$.
\end{proof}

As the previous proposition illustrates, working with the form $\rho$ requires significantly more computational effort. As such, Propositions \ref{prop:G2form is nondegenerate} and \ref{prop:G2form is compatible with x} will be checked primarily with computer algebra packages. These computations were done in Macaulay2; explicit code for each calculation is referred to in Appendix \ref{appendix: Macaulay code}.

\begin{proposition}
	\label{prop:G2form is nondegenerate}
	Let $\nu$ be the bilinear form associated to $\rho$ as in equation \eqref{eqn: associated symmetric form} and let $\omega\in S^2_AB^*$ be the symmetric, nondegenerate form given by the formula \eqref{eqn: odd orthogonal form}. Then, $\nu = -2^4 3^2\omega$.
\end{proposition}

\begin{proposition}
	\label{prop:G2form is compatible with x}
	The form $\rho$ is compatible with the endomorphism $[x]$, in the sense that
	\[
	\rho(xb_1,b_2,b_3)+\rho(b_1,xb_2,b_3)+\rho(b_1,b_2,xb_3) = 0.
	\]
\end{proposition}

As such, the form $\rho$ together with a trivialization of the determinant gives a map $[x]\colon \fc\to [\fg_2/G_2]$.

\section{Special components}

In the previous section, we gave explicit formulas for the tensors defining the reduction of the vector bundle $\cO_{\fc\times_{\fc_n} \fs_n}$ to $G$ so that the companion section for $\GL_n$ induces the companion section for classical group $G$. These explicit formulas may feel like miracles, especially in the $G_2$ case where a computer algebra system is needed. In this section, we will derive them from the geometry of spectral covers, which makes the construction more conceptual, especially in the $G_2$ case. In subsequent work, we use this approach to construct the companion section uniformly. 

\subsection{Special form and component  associated with a subcover} \label{subsection:special form}

Let $A$ be a $k$-algebra, $B$ a finite flat $A$-algebra of degree $n$ generated by one element $b\in B$, and $A'\subset B$ an $A$-subalgebra of $B$ such that $A'$ is finite flat of degree $m$ over $A$ generated by one element $a'\in A'$ and $B$ is a finite flat $A'$-algebra of degree $d$ generated by $b$. Under these assumptions, we have $B\simeq A[x]/P(x)$ where $P(x)$ is the characteristic polynomial of the $A$-linear $b:B\to B$ defined as the multiplication by $b$. Similarly we have $A'\simeq A[x]/(P_1(x))$ where $P_1(x)$ is the the characteristic polynomial of the $A$-linear operator $a':A'\to A'$, and $B\simeq A'[x]/P_2(x)$ where $P_2(x)$ is the characteristic polynomial of the $A'$-linear operator $b:B\to B$. 

Assuming that the characteristic of $k$ is greater than $d$, we want to construct an alternating $d$-form $$\omega_{A'}:\wedge^d_A B\to A$$ supported on a special component of $\Spec(S^d_A B)$ isomorphic to $\Spec(A')$. We explain what this means. As far as we know, the concept of non-degeneracy for $d$-forms is not yet defined for $d\geq 3$ and thus we can prove the it only for $d=1$ or $d=2$. However, we expect that the form we construct is non-degenerate for a reasonable definition of this concept. As to the special component, $\bigwedge^d_A B$ is a module over the ring of symmetric tensors $ (\bigotimes_A^ d B)^{\fS_d}$. We will construct a surjective homomomorphism of $A$-algebras $(\bigotimes_A^ d B)^{\fS_d} \to A'$ which realizes $\Spec(A')$ as an irreducible component of $\Spec((\bigotimes_A^ d B)^{\fS_d})$ if $B$ is generically étale over $A$ and $A'$ is a domain. 

The homomorphism of $A$-algebras $(\bigotimes_A^ d B)^{\fS_d} \to A'$ is constructed as follows. Let $P_2(x)=x^d+a'_1 x^{d-1}+\cdots+a'_d$ be the characteristic polynomial of the $A'$-linear map $b:B\to B$. Then we have
\[B=A'[x]/(x^d+a'_1 x^{d-1}+\cdots+a'_d).\]
We consider the polynomial ring $R=k[x_1,\ldots,x_d]$ and the subring $S$ of invariant polynomials under the symmetric group $\fS_d$. We have \[ S=k[x_1,\ldots,x_d]^{\fS_d}=k[\alpha_1,\ldots,\alpha_d]\] with \[\alpha_i=(-1)^i \sum_{1\leq j_1 < \cdots < j_i \leq d} \alpha_{j_1} \ldots \alpha_{j_d}.\]
Since $R$ and $S$ are regular, and $R$ is a finite generated $S$-module, $R$ is a finite flat $S$-algebra of degree $d!$.   
We consider the homomorphism of algebras $S\to A'$ given by $\alpha_i \mapsto a'_i$ and the base change $A'\otimes_S R$ which is a finite flat $A'$-algebra of degree $d!$ equipped with an action of $\fS_d$. We have $(A'\otimes_S R)^{\fS_d}=A'$.
Moreover, for every $i\in \{1,\ldots,d\}$ we have a homomorphism of $A'$-algebras $B \to R\otimes_S A'$ given by $x\mapsto x_i$ which together give rise to a surjective homomorphism of $A'$-algebras 
$\bigotimes_{A'}^d B \to A'\otimes_S R$, which is $\fS_d$-equivariant.  
We derive a $\fS_d$-equivariant surjective homomorphism of $A$-algebras
\begin{equation} \label{special-component}
	\bigotimes_{A}^d  B\to \bigotimes_{A'}^d  	B\to A'\otimes_S R.
\end{equation}
By taking the $\fS_d$-invariant, we obtain the desired homomorphism of algebras $$S^d_A B=(\bigotimes_A^ d B)^{\fS_d} \to A',$$ which is surjective because taking $\fS_d$-invariants is an exact functor under the characteristic assumption.

We will now construct a special $d$-form on $B$
\[\omega_{A'}:\bigwedge^d_A B \to A\] 
supported on the special component. As above, we have a surjective homomorphism of algebras $\fS_d$-equivariant surjective homomorphisms of $A$-algebras $\otimes^d_A B \to \otimes^d_{A'}B \to A'\otimes_S R$ which induces a surjective $A$-linear maps of the alternating parts 
$\bigwedge^d_A B \to \bigwedge^d_{A'} B \to A'\otimes_S R^{\sgn}$ where $R^{\sgn}$ is the direct factor of $R$ as $S$-module in which $\fS_d$ acts as the sign character. It is known that $R^{sgn}$ is a free $S$-module generated by $\prod_{1\leq i < j \leq d} (x_i-x_j)$. We thus obtains a surjective $A$-linear map $\bigwedge_A^d B \to A'$. By composing it with the generator of $\Hom_A(A',A)$ constructed in \ref{lem:beta} we obtain the special $d$-form $\omega_{A'}:\bigwedge^d_A B \to A$ which is supported by the special by construction.

Let us discuss the non-degeneracy of the special $d$-form $\omega_{A'}:\bigwedge^d_A B \to A$. For $d=1$, this follows from Lemma \ref{lem:beta}. We can check that it is also non-degenerate everywhere for $d=2$. For $d\geq 3$, we don't know a general definition of non-degeneracy but it easy to see that the special form $\omega_{A'}$ is everywhere non-zero.  In dimension 6 and 7 where the definition of non-degeneracy is available, we will check that the special $d$-form is everywhere non-degenerate by direct calculation. 

\subsection{$\Sp_{2n}$ case}

We recall in the case $G=\Sp_{2n}$, we have $\fc=\Spec(A)$ with $A=k[a_2,\ldots,a_{2n}]$. The spectral cover $\mathfrak{s}=\Spec(B)$ where  \[B=A[x]/(x^{2n}+a_2 x^{2n-2} +\cdots +a_{2n})\] is a free $A$-module of rank $2n$, is equipped with an involution $\tau:B\to B$ given $\tau(x)=-x$. We consider the subalgebra $A'$ of $B$ consisting of elements fixed under $\tau$
\[ A'=A[y]/(y^n+a_2 y^{n-1}+\cdots +a_{2n}).\]
We then have $B=A'[x]/(x^2-y)$. 

The construction of the special form and special component in \ref{subsection:special form} gives rise to an alternating form
\[ \omega_{A'}: \wedge^2_A B\to A\]
supported in the special component $\fc'=\Spec(A')$ of $(\fs\times_\fc \fs)\git \fS_2$ where $\fs=\Spec(B)$ and $\fc=\Spec(A)$. The homomorphism \eqref{special-component} $\Sym_A^2(B) \to A'$ can be explicitly computed elements of the form:
\[ b\otimes_A 1 + 1 \otimes_A b \mapsto \tr_{B/A'}(b).\]
In particular, $x\otimes_A 1+ 1\otimes_A x$ be long to the kernel of $\Sym_A^2(B) \to A'$, and in fact on can verify that it is a generator of the kernel. Since $x\otimes_A 1+ 1\otimes_A x$ annihilates $\omega_{A'}$ we have 
\[ \omega_{A'}(xb_1,b_2)+ \omega_{A'}(b_1,xb_2)=0\]
for every $b_1,b_2\in B$. By Lemma \ref{lem:beta}, the 2-form $\omega_{A'}$ is everywhere non-degenerate. We can also see by explicit calculation that the form $\omega_{A'}$ is the same as the 2-form we constructed in subsection \ref{subsection:symplectic} by means of the Euler formula. 

\subsection{$G_2$ case}
\label{sec: G2 case theoretical}

In the case $G_2$, the invariant quotient is $A = k[e,q]$ with $\deg(e) = 2$ and $\deg(q) = 6$. The spectral cover $\fs = \Spec(B)$ with
\[
B = A[x]/\left(xf_0\right)\quad\text{for }f_0 = x^6-ex^4+\frac{e^2}{4}x^2+q
\]
is a reducible cover of $A$ with two components corresponding to the quotient maps
\[
B\to B' = A[x]/(f_0)\quad\text{and}\quad B\to A
\]
We will define a canonical 3-form on $B$ out of a 3-form and a 2-form on $B'$ associated to subalgebras  
\begin{eqnarray*}
	A\subset A'=A[z]/\left( z^2+q \right)=k[e,y]& \subset & B'=A'[x]/(x^3-\frac{e}{2} x-z) \\
	A\subset  A''=A[y]/\left( y^3 -e y^2 +\frac{e^2}{4} y +q \right)=k[e,z]&\subset & B'= A''[x]/(x^2-y).
\end{eqnarray*}
Since both $A'$ and $A''$ are regular algebras, they are finite flat $A$-modules of rank 2 and 3, respectively, whereas $B$ are finite flat $A'$-module and $A''$-module of rank 3 and 2, respectively. The construction of the special form associated with a subcover gives rise to 
\[ \omega_{A'}:\wedge^3_A B'\to A \mbox{ and } \omega_{A''}:\wedge^2_A B'\to A\]
supported on the special components $\fc'=\Spec(A')$ and $\fc''=\Spec(A'')$ of $\fs^{\times_\fc^3}\git \fS_3$ and $\fs^{\times_\fc^2}\git \fS_2$, respectively. By arguing as in the symplectic case, we see that $\omega_{A'}$ is annihilated by $x\otimes_A 1 \otimes_A 1+ 1\otimes_A x \otimes_A 1+ 1\otimes_A 1 \otimes_A x$ and $\omega_{A'}$ by $x\otimes_A 1+ 1\otimes_A x$. It follows that as alternating forms, they satisfy the relations:
\begin{eqnarray*}
	\omega_{A'}(xb_1,b_2,b_3)+ \omega_{A'}(b_1,xb_2,b_3) + \omega_{A'}(b_1,b_2,xb_3)&=&0 \\
	\omega_{A''}(xb_1,b_2) + \omega_{A''}(b_1,xb_2)&= &0
\end{eqnarray*}
for all $b_1,b_2,b_3\in B$. 

The form $\omega_{A'}$ agrees with the restriction of the form $\rho$ calculated by Macaulay 2 when restricted to $B'\to B$, with the inclusion given by multiplication by $x$:  Indeed, the restriction of $\rho$ takes value 1 on each of:
$$z\wedge x\wedge x^2,\quad x\wedge zx\wedge x^2,\quad 1\wedge x\wedge zx^2$$
and $-q$ on $z\wedge zx\wedge zx^2$. This exactly detects the coefficient of $z$ when these wedges are written in terms of the $A'$ basis $1\wedge x\wedge x^2$ for $\wedge^3_{A'}B'$, which matches $\omega_{A'}$ since the generator of $\Hom_A(A',A)$ as an $A''$ module detects the coefficient of $A'$.

We now build a 3-form on $B$ out of the 3-form $\omega_{A'}$ and 2-form $\omega_{A''}$ on $B'$. Since $B=A[x]/(xf_0)$, $B'=A[x]/(f_0)$ we have exact sequences of free $A$-modules
\[ 0 \to A \to B \to B' \to 0 \mbox{ and } 0 \to B'\to B \to A \to 0\]
where the map $A\to B$ in the first sequence is given by $1\mapsto f_0$ and the map $B'\to B$ in the second sequence is given by $1\mapsto x$. It follows an exact sequences
\[ 0 \to  A \oplus B' \to B \to Q \to 0 \mbox{ and } 0 \to B \to A\oplus B' \to Q \to 0 \]
where $Q=A/(q)=B'/(x)$. It follows an exact sequence
\[ 0 \to \wedge_A^3 B^* \to \wedge^3_A (B')^* \oplus \wedge^2_A (B')^* \to \wedge^2 (B')^*/(q) \to 0\]
where the map $\wedge^2_A (B')^* \to \wedge^2 (B')^*/(q)$ is the reduction modulo $q$, and the map $\wedge^3_A (B')^* \to \wedge^2 (B')^*/(q)$ is obtained by the composition
\[
\wedge^3_A(B')^*\to \wedge^3_AB^*\to \wedge^2_AB^*\to \wedge^2_A(B')^*\to \wedge^2_A(B')^*/(q)
\]
where the first map is induced by the projection $B\to B'$, the second is given by contraction with $f_0$, the third map is induced by the inclusion $B'\to B$ sending $1\mapsto x$, and the final map is the quotient map. Since $q\wedge^2(B')^*\simeq \wedge^2(B')^*$ is a free, rank 1 module over the special component of $S^2_A(B')$, there is a unique generator as an $A''$ module. The 3-form $\omega_{A'}$ and the 2-form $\omega_{A''}$ do not have the same image in $\wedge^2_A(B')^*/(q)$; however, the form $z\omega_{A''}$ is and it gives a generator for the $A''$ submodule of 2-forms compatible with $\omega_{A'}$. The pair $(\omega_{A'},z\omega_{A''})$ comes from an element of $\wedge_A^3B^*$ which agrees with the 3-form calculated by Macaulay2.

\section{Lattice description of affine Springer fibers of classical groups}
\label{sec: affine springer}

Let us recall Kazhdan-Lusztig's definition \cite{kl} of affine Springer fibers. Let $G$ be a split reductive group defined over a field $k$ and $\fg$ its Lie algebra. Let $F=k((\varpi))$  the field of Laurent formal series and $\cO=k[[\varpi]]$ its ring of integers. Let $\gamma\in \fg(F)$ be a regular semisimple element. The affine Springer fiber associated with $\gamma$ is an ind-scheme defined over $k$ whose set of $k$-points is 
$$\cM_\gamma(k)=\{g\in G(F)/G(\cO) | \ad(g)^{-1} \gamma \in \fg(\cO)\}.$$
We note that $\cM_\gamma$ is non-empty only if the image $a\in \fc(F)$ lies in $\fc(\cO)$ where $\fc=\fg\git G$ is the invariant theoretic quotient of $\fg$ by the adjoint action of $G$. As argued in \cite{ngo}, using the Kostant section, we can define an affine Springer fiber $\cM_a$ depending only on $a$ instead of $\gamma$, which is isomorphic to $\cM_\gamma$. 

For $G=\GL_n$, the affine Springer fiber $\cM_a$ has a well-known lattice description. In this case, $\fc=\bA^n$. If $a=(a_1,\ldots,a_n)\in \cO_n$, we form the finite flat $\cO$-algebra
$$B_a=\cO[x]/(f_a)$$
where $f_a=x^n+a_1 x^{n-1}+\cdots +a_n$ by the base change from the universal spectral cover. As $\gamma\in \fg(F)$ is a regular semisimple element, $B_a\otimes_\cO F$ is finite and étale over $F$.  We have a well-known lattice description of the affine Springer fiber $\cM_a$ in this case.
\begin{theorem}
	\label{thm:lattice desription}
	For $G = \GL_n$ and $a\in \fc^{\mathrm rs}(F)\cap \fc(\cO)$, the set $\cM_a(k)$ consists of lattices $\cV$ in the $n$-dimensional vector space $V=B_a\otimes F$ which are also $B_a$-modules. 
\end{theorem}
See for example, Section 2 of \cite{yun} for an exposition.

For computational purposes, it is desirable to have a lattice description of affine Springer fibers similar to Theorem \ref{thm:lattice desription} for classical groups, which is as simple as in the linear case. This is possible due to the construction of the companion matrix, and in fact, this was our original motivation. 

In the cases we have investigated in the paper, i.e., symplectic, special orthogonal, and $G_2$, we have constructed a finite, flat spectral cover $\fs=\Spec(B)$ of the invariant theoretic quotient $\fc=\Spec(A)$ which is étale over the regular semisimple locus of $\fc$. The degree $d=\deg(B/A)$ is the degree of the standard representation which is $2n$ for $\Sp_{2n}$, $2n+1$ for $\SO_{2n+1}$, $2n$ for $\SO_{2n}$ and $7$ for $G_2$. In the case $\SO(2n)$, we must consider the normalization $\tilde B$ of $B$. In each of these cases, we constructed a form $\omega$, which is 
\begin{itemize}
	\item a non-degenerate symplectic form $\omega:B\times B\to A$ satisfying $\omega(xb_1,b_2)+\omega(b_1,xb_2)=0$ for $\Sp_{2n}$
	\item a non-degenerate symmetric form $\omega:B\times B\to A$ satisfying $\omega(xb_1,b_2)+\omega(b_1,xb_2)=0$ for $\SO_{2n+1}$
	\item a non-degenerate symmetric form $\omega:\tilde B\times \tilde B\to A$ satisfying $\omega(xb_1,b_2)+\omega(b_1,xb_2)=0$ for $\SO_{2n}$
	\item a non-degenerate alternating form $\omega:B\times B \times B\to A$ satisfying $$\omega(xb_1,b_2,b_3)+\omega(b_1,xb_2,b_3)+\omega(b_1,b_2,xb_3)=0$$ for $G_2$
\end{itemize}
We also constructed a trivialization of the determinant $\bigwedge^d_A B=A$ in all these cases. 

For every $a\in \fc(\cO)\cap \fc^{rs}(F)$, we construct a finite flat $\cO$-algebra $B_a$ by base change from the spectral cover $\fs\to \fc$. Because $a\in \fc^{rs}(F)$, the generic fiber $V_a=B_a\otimes_\cO F$ is a finite étale $F$-algebra of degree $d$. By pulling back $\omega$, we get a form $\omega_a$ which is a non-degenerate alternating $F$-bilinear form on $V_a$ in the symplectic case, a non-degenerate symmetric $F$-bilinear form on $V_a$ in the orthogonal case, and a non-degenerate alternating $F$-trilinear form on $V_a$ in the $G_2$ case. Moreover, it extends to a non-degenerate form valued in $\cO$ on $B_a$ in $\Sp_{2n}$, $\SO_{2n+1}$ and $G_2$ cases and on $\tilde B_a$ in the $\SO_{2n}$-case. 

\begin{theorem}
	The set of $k$-points of the affine Springer fiber $\cM_a$ is the set of $\cO$-lattices $\cV$ of $V_a$, which are $B_a$-modules,  such that the restriction of $\omega_a$ has value in $\cO$ and such that $\deg(\cV:B_a)=0$ in $\Sp_{2n}, \SO_{2n+1}, G_2$ cases and $\deg(\cV:\tilde B_a)=0$ in the $\SO_{2n}$ case. 
\end{theorem}

The proof of this result follows immediately from the proof of Theorem \ref{thm:lattice desription}, as lattices preserved by the nondegenerate form $\omega_a$ constructed above are exactly those for which there is a reduction of structure to the classical group $G$.

\section{Application to the Hitchin fibration}

Let $X$ be a smooth, projective curve over an algebraically closed field $k$ and let $G$ be a reductive group over $k$ with Lie algebra $\fg$. Fix a line bundle $L$ on $X$ such that either $\deg(L)>2g-2$ or $L = K$ is the canonical bundle. Denote by $\cM$ the moduli stack of Higgs bundles on $X$, whose $k$ points are given by the set of Higgs bundles
$$\cM(k) = \{(E,\phi)\colon E\to X\text{ is a $G$ bundle, }\phi\in \Gamma(X,{\rm ad}(E)\otimes L)\}$$
More succinctly, $\cM$ is the mapping stack $\cM = {\rm Maps}(X,[\fg_L/G])$ where $\fg_L = \fg\wedge^{\bG_m}L$ is the twisted bundle of Lie algebras on $X$. 

Recall that under mild hypotheses on the characteristic of $k$ (${\rm char}(k)>2$ for $G = \SO_n$ and $\Sp_{2n}$ and ${\rm char}(k)>3$ for $G = G_2$), the Chevalley isomorphism shows 
\[
\fg\git G\simeq \ft\git W\simeq \bA^n
\]
is an affine space with $\bG_m$ action by weights $d_1,\dots, d_n$. Let 
\[
\cA = {\rm Maps}(X,\fg_L\git G) \simeq \otimes_{i=1}^n \Gamma(X,L^{\otimes d_i})
\]
Hitchin, in \cite{hitchin}, studied the space $\cM$, with appropriate stability conditions imposed, through the fibration that now bears his name:
\[
h\colon \cM\to \cA, \quad (E,\phi)\mapsto {\rm char}(\phi)
\]
where ${\rm char}(\phi)$ is given by composition with the quotient map $[\fg/G]\to \fg\git G$. Let $\cM_a$ denote the fiber of the map $h$ over a point $a\in \cA$. In the case that $G = \GL_n$, $d_i = i$ and ${\rm char}(\phi) = \sum_i a_i x^i$ is the characteristic polynomial of $\phi$, whose coefficients are then sections $a_i\in \Gamma(X,L^{\otimes i})$. 

The companion section $[x]:\fg\git G\to [\fg/G]$ can be used to construct an explicit section to the Hitchin map after extracting a square root of $L$. This section in many cases is almost the same as the section constructed by Hitchin \cite{hitchin} and \cite{hitchin-G2}, but can be different from the section constructed in \cite{ngo} which is based on the Kostant section. In every case, the Higgs bundle constructed from the companion section will be built out of the structural sheaf of the spectral curve. Note that the following assumes basic $\bG_m$ equivariance properties of the relevant forms. For example, in the case of $G = \Sp_{2n}$, we have constructed a canonical alternating form $\omega\colon \wedge^2_A B\to A$ which satisfies $\omega(\lambda \xi) = \lambda^{1-2n}\omega(\xi)$ for any $\lambda \in \bG_m$ and $\xi\in \wedge^2_AB$.

In \cite{ngo}, it is shown that over a large open subset of $\cA$, there is a close connection, depending on a choice of section, between Hitchin fibers and affine Springer fibers given by the Product Formula. More precisely, let $\mathfrak{D} = \bigcup_{\alpha} \ft^{s_\alpha}\git W$ be the divisor consisting of the union of the image of each root hyperplane in $\ft$; in particular, the complement of $\mathfrak{D}$ in $\fc$ is the regular, semisimple locus $\fc^{rs}$. Fix $a\in \cA$ such that $a(X)\not\subset \mathfrak{D}$, and let $U\subset X$ be the preimage of $\fc^{rs}$ in $X$. Given trivialization of the line bundle $D$ on some neighborhood of each point $v\in X\setminus U$, we have a map
\[
\prod_{v\in X\setminus U}\cM_{x,a}\to \cM_a.
\]
from the product of affine Springer fibers at the points $x\in X\setminus U$ to the Hitchin fiber, which consists of gluing with the companion section restricted to $U$. It it induces a universal homeomorphism
\[
\prod_{\gamma\in X\setminus U}\cM_{\gamma,a}\wedge^{\prod_\gamma \cP_{\gamma}(J_a)}\cP_a\to \cM_a.
\]
The groups $\cP_\gamma(J_a)$ and $\cP_a$ are discussed in detail in \cite{ngo}; we will not describe them here. This is proved in \cite{ngo} under the assumption that $\pi_0(\cP_a)$ is finite, and  by Bouthier and Cesnavicius in \cite{bouthier} under the only assumption that $a(X)\not\subset \mathfrak{D}$.

As Section \ref{sec: affine springer} describes the affine springer fibers $\cM_{\gamma,a}$, the product formula above gives an explicit description of Hitchin fibers in the case that $a(X)\not\subset \mathfrak{D}$. Namely, we have the following descriptions for Hitchin fibers under this assumption.
\begin{itemize}
	\item for $G=\GL_n$, and $a\in \cA$ we have a spectral cover $p_a:Y_a \to X$ embedded in the total space $|L|$ of $L$. We then associate with $a$ the Higgs bundle $E_a=p_{a*} \cO_{Y_a}$ and the Higgs fields $\phi:E_a\to E_a\otimes L$ given by the structure of $\cO_{Y_a}$ as an $\cO_{|L|}$-module. 
	\item for $G=\Sp_{2n}$, and $a\in \cA$, we have a spectral cover $p_a:Y_a \to X$ embedded in the total space $|L|$ of $L$. If $E_a=p_{a*} \cO_{Y_a}$ then we have a canonical symplectic form $\wedge^2 E_a \to L^{\otimes(1-2n)}$. If $L'$ is a square root of $L$ then $E'_a = E_a\otimes {L'}^{\otimes 1-2n}$ will be equipped with a canonical symplectic form with value in $\cO_X$ and also equipped with a Higgs fields derived from the the structure of $\cO_{Y_a}$ as a $\cO_{|L|}$-module.
	\item for $G=\SO_{2n+1}$, and $a\in \cA$, we have a spectral cover $p_a:Y_a \to X$ embedded in the total space $|L|$ of $L$. If $E_a=p_{a*} \cO_{Y_a}$ then we have a canonical non-degenerate symmetric form $S^2 E_a \to L^{\otimes(-2n)}$ so that the vector bundle $E'_a=E_a \otimes L^{\otimes n}$ affords a canonical no-degenerate symmetric form with value in $\cO_X$, and also equipped with a Higgs fields derived from the the structure of $\cO_{Y_a}$ as a $\cO_{|L|}$-module. It also affords a trivialization of the determinant depending on the choice of a square root of $L$. 
	\item for $G=\SO_{2n}$, and $a\in \cA$, we have a spectral cover $p_a:Y_a \to X$ embedded in the total space $|L|$ of $L$. Using the normalization of the universal spectral cover, we obtain a partial normalization $\tilde Y_a$ of $Y_a$. If $E_a=p_{a*} \cO_{\tilde Y_a}$ then we have a canonical non-degenerate symmetric form $S^2 E_a \to L^{\otimes(2-2n)}$ so that the vector bundle $E'_a=E_a \otimes L^{\otimes 1- n}$ affords a canonical non-degenerate symmetric form with values in $\cO_X$, and also equipped with a Higgs fields derived from the the structure of $\cO_{Y_a}$ as a $\cO_{|L|}$-module. It also affords a canonical trivialization of the determinant depending on the choice of a square root of $L$. 
	\item for $G=G_2$, and $a\in \cA$, we have a spectral cover $p_a:Y_a \to X$ embedded in the total space $|L|$ of $L$. If $E_a=p_{a*} \cO_{\tilde Y_a}$ then we have a canonical non-degenerate 3-form $\wedge^3 E_a \to L^{-9}$ so that the vector bundle $E'_a=E_a \otimes L^{\otimes 3}$ affords a canonical non-degenerate 3-form with value in $\cO_X$, and also equipped with a Higgs fields derived from the the structure of $\cO_{Y_a}$ as a $\cO_{|L|}$-module. It also affords a canonical trivialization of the determinant depending on the choice of a square root of $L$. 
\end{itemize}

\appendix

\section{Computer algebra code and $G_2$ computations}
\label{appendix: Macaulay code}

In this appendix, we give the computer code used to compute the 3-form $\rho$ in Section \ref{sec: G2 case}.

\subsection{Construction of $\rho$}

To construct $\rho$, we will use the connection between nondegenerate alternating 3-forms and cross products. Let $V$ be a vector space with a nondegenerate, symmetric bilinear form $\nu$.
\begin{definition}
	\label{cross product}
	A \emph{cross product} on $(V,\nu)$ is a bilinear map
	\[
	c\colon V\otimes V\to V
	\]
	satisfying the following three properties for all $v_1,v_2\in V$:
	\begin{enumerate}
		\item (Skew symmetry) $c(v_1,v_2) = -c(v_2,v_1)$;
		\item (Orthogonality) $\nu(c(v_1,v_2),v_1) = 0$;
		\item (Normalization) $\displaystyle \nu\big(c(v_1,v_2),c(v_1,v_2)\big) = \det\begin{pmatrix}
			\nu(v_1,v_1) & \nu(v_1,v_2) \\ \nu(v_1,v_2) & \nu(v_2,v_2)
		\end{pmatrix}$
	\end{enumerate}
\end{definition}

The data of a cross product on $(V,\nu)$ is equivalent to the data of a nondegenerate 3-form on $V$ whose associated symmetric bilinear form (see equation \eqref{eqn: associated symmetric form}) is a scalar multiple of $\nu$. Indeed, to a cross product $c$, one associates the 3-form
\begin{equation}
	\label{eqn: 3form to cross product}
	\rho(v_1,v_2,v_3) = \nu\big( c(v_1,v_2),v_3 \big)
\end{equation}
while for any non-degenerate 3-form $\rho$, there is a unique cross product $c$ satisfying equation \eqref{eqn: 3form to cross product}. 

Now, consider the free, rank 7 $A$-module $B$ as in Section \ref{sec: G2 case} equipped with the symmetric, nondegenerate form $\omega$ defined by the formula
\[
\omega(g_1,g_2) = {\rm tr}_{B/A}\left( \frac{g_1\tau(g_2)}{f'} \right)
\]
as in the $\SO_{7}$ case. Here, $\tau(x) = -x$ is the natural involution on $B$, and the trace is taken after inverting $f'$ in $A$. To construct a 3-form on $B$ which is nondegenerate over every $k$ point of $A$, it suffices to construct a cross product
\[
c\colon B\otimes_AB\to B
\]
for $(B,\omega)$. Moreover, the equation
\[
\rho(xg_1,g_2,g_3)+\rho(g_1,xg_2,g_3)+\rho(g_1,g_2,xg_3) = 0
\]
is equivalent to the condition
\begin{equation}
	\label{eqn: cross product compatibility with x}
	c(xg_1,g_2)+c(g_1,xg_2) = xc(g_1,g_2).
\end{equation}
To simplify computations further, we note that any form $c\colon B\otimes_A B\to B$ satisfying the conditions of Definition \ref{cross product} and equation \eqref{eqn: cross product compatibility with x} can be recovered from its trace:
\[
tc\colon B\otimes_A B\to A, \quad (g_1,g_2)\mapsto {\rm tr}_{B/A}\big( c(g_1,g_2) \big)
\]
Indeed, if we express
\[
c(x^i,x^j) = \sum_{l=0}^6 c_{i,j}^{(l)} x^l
\]
then $c_{i,j}^{(6)} = tc(x^i,x^j)$ and
\[
tr_{B/A}\big(x^l c(x^i,x^j)\big) = \sum_{r=0}^l\binom{l}{r}tc(x^{i+r},x^{j+l-r})
\]
can be expressed in terms of $c_{i,j}^{(m)}$ for $6-l\leq m\leq 6$. This allows us to recover the coefficients $c_{i,j}^{(l)}$ by downward induction on $l$.

This idea is implemented in the following Macaulay2 code. There is a one-dimensional solution space, which is specialized at a particular point to give the form stated in equation \eqref{eqn:G2form}. Note that it is immediate from the computer calculation that the form $\rho$ is valued in $B$ and satisfies the conclusion of Proposition \ref{prop:G2form is compatible with x}.

\begin{verbatim}
	S=QQ[e,q];
	F=frac(S);
	R=F[p_(0,0) .. p_(6,6)]; -- ring with p_(i,j)=tc(x^i,x^j), 
	0\leq i,j\leq 6
	
	-- The following three commands define tc(x^i,x^j) for i or j between 
	7 and 12 using the relation x^7-e*x^5+e^4/4*x^3+q*x=0.
	for l from 0 to 5 do [for k from 0 to 6 do p_(k,7+l)=e*p_(k,5+l)-
	(1/4)*e^2*p_(k,3+l)-q*p_(k,1+l)];
	for l from 0 to 5 do [for k from 0 to 6 do p_(7+l,k)=e*p_(5+l,k)-
	(1/4)*e^2*p_(3+l,k)-q*p_(1+l,k)];
	for l from 0 to 5 do [for k from 7 to 12 do p_(k,7+l)=e*p_(k,5+l)-
	(1/4)*e^2*p_(k,3+l)-q*p_(k,1+l)];  
	
	-- I encodes orthogonality:
	I = ideal(flatten for a from 0 to 6 list for k from 0 to 6 list 
	sum(0..k,j->binomial(k,j)*p_(k+j,a+k-j)));
	
	-- J encodes skew symmetry:
	J = ideal( flatten for a from 0 to 6 list for b from 0 to 6 list 
	p_(a,b)+p_(b,a) );
	
	-- The following encodes the normalization condition:
	B=R[x]/(x^7-e*x^5+(1/4)*e^2*x^3+q*x);
	-- determinant of norms of x^i,x^j:
	f = (i,j) -> coefficient(x^6,(-1)^i*x^(2*i))*coefficient(x^6,(-1)^j*
	x^(2*j))-coefficient(x^6,(-1)^j*x^(i+j))*coefficient(x^6,(-1)^j*
	x^(i+j));
	-- norm of c(x^i,x^j):
	g = (i,j) -> coefficient(x^6, (p_(i,j)*(x^6-e*x^4+(1/4)*e^2*x^2+q)+
	sum(0..1,l->binomial(1,l)*p_(i+l,j+1-l))*(x^5-e*x^3+(1/4)*e^2*x)+
	sum(0..2,l->binomial(2,l)*p_(i+l,j+2-l))*(x^4-e*x^2+(1/4)*e^2)+
	sum(0..3,l->binomial(3,l)*p_(i+l,j+3-l))*(x^3-e*x)+sum(0..4,l->
	binomial(4,l)*p_(i+l,j+4-l))*(x^2-e)+sum(0..5,l->binomial(5,l)*
	p_(i+l,j+5-l))*(x)+sum(0..6,l->binomial(6,l)*p_(i+l,j+6-l)))
	*(p_(i,j)*((-x)^6-e*(-x)^4+(1/4)*e^2*(-x)^2+q)+sum(0..1,l->
	binomial(1,l)*p_(i+l,j+1-l))*((-x)^5-e*(-x)^3+(1/4)*e^2*(-x))+
	sum(0..2,l->binomial(2,l)*p_(i+l,j+2-l))*((-x)^4-e*(-x)^2+(1/4)*
	e^2)+sum(0..3,l->binomial(3,l)*p_(i+l,j+3-l))*((-x)^3-e*(-x))+
	sum(0..4,l->binomial(4,l)*p_(i+l,j+4-l))*((-x)^2-e)+sum(0..5,l->
	binomial(5,l)*p_(i+l,j+5-l))*(-x)+sum(0..6,l->binomial(6,l)*
	p_(i+l,j+6-l))) );
	-- K encodes the normalization condition:
	K = ideal(flatten for i from 0 to 6 list for j from 0 to 6 list 
	f(i,j)-g(i,j)); 
	
	Q=R/(I+J+K); -- imposing the relations on our ring of variables
	Q2=Q/ideal(p_(6,3)-1,p_(6,4),p_(6,5)-5*e/2); -- specializes to our 
	particular form rho
	
	-- Computation of c from tc:
	P=Q2[x]/(x^7-e*x^5+e^2/4*x^3+q);
	C=table(for k from 0 to 6 list k, for k from 0 to 6 list k, (i,j) -> 
	(p_(i,j)*(x^6-e*x^4+(1/4)*e^2*x^2+q)+sum(0..1,l->binomial(1,l)*
	p_(i+l,j+1-l))*(x^5-e*x^3+(1/4)*e^2*x)+sum(0..2,l->binomial(2,l)*
	p_(i+l,j+2-l))*(x^4-e*x^2+(1/4)*e^2)+sum(0..3,l->binomial(3,l)*
	p_(i+l,j+3-l))*(x^3-e*x)+sum(0..4,l->binomial(4,l)*p_(i+l,j+4-l))*
	(x^2-e)+sum(0..5,l->binomial(5,l)*p_(i+l,j+5-l))*(x)+sum(0..6,l->
	binomial(6,l)*p_(i+l,j+6-l)))); 
	-- This is the matrix for c with respect to the basis x^i, i=0,..,6
	netList C -- displays C
\end{verbatim}

\subsection{Nondegeneracy of $\rho$}

Let $\rho$ be the form computed in the previous section, stated explicitly in equation \eqref{eqn:G2form}. Note that since we specialized to a particular form in the previous section, it is not yet clear that this form is nondegenerate. For this, we produce the following code in Macaulay2 to explicitly compute the associated bilinear form is as in Proposition \ref{prop:G2form is nondegenerate}. The following uses some basic operations on permutations from the package SpechtModule authored by Jonathan Niño in Macaulay2.

\begin{verbatim}
	T=permutations {0,1,2,3,4,5,6};
	n = (v,w) -> sum(0..7!-1, k-> permutationSign(T_k)*coefficient(x^6,
	v*(C_((T_k)_0))_((T_k)_1))*coefficient(x^6,w*(C_((T_k)_2))_((T_k)_3))
	*coefficient(x^6,(-x)^((T_k)_4)*(C_((T_k)_5))_((T_k)_6))  );
	S=table(for k from 0 to 6 list k, for k from 0 to 6 list k, (i,j) -> 
	n((-x)^i,(-x)^j);
	netList S
\end{verbatim}

\bibliographystyle{amsalpha}
\bibliography{companion/example.bib}
\end{document}